\documentclass[10pt,a4paper]{amsart}

\usepackage[english]{babel}
\usepackage{amsmath,amsthm,amssymb}
\usepackage{mathtools}
\usepackage[%
 pagebackref = true,
  colorlinks = true,
   pdfauthor = {Dimitrios Chatzakos and Giacomo Cherubini and Niko Laaksonen},
    pdftitle = {Second Moment of the Prime Geodesic Theorem for the Picard Manifold},
  pdfcreator = {Pdflatex},
 pdfkeywords = {prime geodesic},
 pdfpagemode = UseNone,
pdfstartview = FitH
]{hyperref}

\numberwithin{equation}{section}

\theoremstyle{plain}
\newtheorem{theorem}{Theorem}[section]
\newtheorem{lemma}[theorem]{Lemma}
\newtheorem{proposition}[theorem]{Proposition}

\theoremstyle{definition}
\newtheorem{remark}[theorem]{Remark}

\def\eps{\epsilon}

\def \Z {{\mathbb Z}}
\def \Q {{\mathbb Q}}
\def \R {{\mathbb R}}
\def \C {{\mathbb C}}

\newcommand{\integers}{\mathbb{Z}}
\newcommand{\complexes}{\mathbb{C}}
\newcommand{\rationals}{\mathbb{Q}}

\DeclareMathOperator{\li}{li}

\let\abs\undefined
\DeclarePairedDelimiter{\abs}{\lvert}{\rvert}

\newcommand{\uphs}{\mathbb{H}^{3}}
\newcommand{\hmodgs}{\Gamma\backslash\uphs}
\newcommand{\zin}{\mathbb{Z}[i]\setminus\{0\}}
\newcommand{\pslc}{\mathrm{PSL}(2, \complexes)}
\newcommand{\pslzi}{\mathrm{PSL}(2, \integers[i])}

\title[Second Moment of the PGT]{Second Moment of
the Prime Geodesic Theorem for $\mathrm{PSL}(2, \mathbb{Z}[i])$}

\author{Dimitrios Chatzakos}
\address{
    Universit\'e de Lille 1 Sciences et Technologies
    and
    Centre Européen pour les Mathématiques, la Physique et leurs interactions (CEMPI),
    Cit\'e Scientifique, 59655 Villeneuve d’ Ascq C\'edex, France
    }
\email{Dimitrios.Chatzakos@math.univ-lille1.fr}

\author{Giacomo Cherubini}
\address{
    Alfr\'ed R\'enyi Institute of Mathematics,
    Hungarian Academy of Sciences,
    POB 127, Budapest H-1364, Hungary;
    MTA R\'enyi Int\'ezet Lend\"ulet Automorphic Research Group
    }
\email{cherubini.giacomo@renyi.mta.hu}

\author{Niko Laaksonen}
\address{
    Alfr\'ed R\'enyi Institute of Mathematics,
    Hungarian Academy of Sciences,
    POB 127, Budapest H-1364, Hungary;
    MTA R\'enyi Int\'ezet Lend\"ulet Automorphic Research Group
    }
\email{laaksonen.niko@renyi.mta.hu}

\thanks{
    The first author is currently supported by the Labex CEMPI (ANR-11-LABX-0007-01).
    He also wishes to thank the Mathematics department of King's College London
    for the hospitality during the spring and the summer of 2017 and the financial support
    through the European Union’s Seventh Framework Programme (FP7/2007-2013) / ERC grant agreement no. 335141 \emph{Nodal}.
    The second and third author were supported by the MTA R\'enyi Int\'ezet Lend\"ulet Automorphic Research Group.
    The second author was also partially supported by a \emph{Ing. Giorgio Schirillo}
    postdoctoral grant (2017--2018) from INdAM.
    The third author would also like to thank the Department of Mathematics and
    Statistics at McGill University and the Centre de Recherches Math\'ematiques
    for their hospitality and support.
    }
\date{\today}

\begin{document}

\begin{abstract}
The remainder $E_\Gamma(X)$ in the Prime Geodesic Theorem
for the Picard group $\Gamma = \mathrm{PSL}(2,\mathbb{Z}[i])$
is known to be bounded by $O(X^{3/2+\eps})$
under the assumption of the Lindel\"of hypothesis
for quadratic Dirichlet $L$-functions over Gaussian integers.
By studying the second moment of $E_\Gamma(X)$, we show that
on average the same bound holds unconditionally.
\end{abstract}

\maketitle


\section{Introduction}

Let $M = \Gamma \backslash \mathbb{H}^2$ be a hyperbolic surface,
with $\Gamma$ a cofinite Fuchsian group, and denote by $\pi_{\Gamma}$ the counting
function of the primitive length spectrum of $M$, i.e.~$\pi_{\Gamma}(X)$ is
the number of primitive closed geodesics on $M$ of length at most $\log X$.
The study of $\pi_{\Gamma}$ has a long history dating back to works of
Huber~\cite{huber_zur_1961,huber_zur_1961-1},
Selberg~\cite[Chapter~10]{iwaniec_spectral_2002}
and others.
In particular, for surfaces of arithmetic type, much progress has been made
in estimating the asymptotic error term related to $\pi_{\Gamma}$, see
e.g.~\cite{iwaniec_prime_1984},
\cite{luo_quantum_1995}, \cite{soundararajan_prime_2013}.
In three dimensions, that is $M=\hmodgs$ where
$\Gamma\subset\pslc$ is a cofinite Kleinian group, we know that~\cite{gangolli_1980}
\begin{equation}\label{eq:pibound}
    \pi_{\Gamma}(X)\sim \li(X^{2}).
\end{equation}
In analogy with the classical prime number theory, it is more convenient to
work with the hyperbolic
analogue of the Chebyshev function, which is defined as
\[
\psi_{\Gamma}(X)=\sum_{N(P)\leq X}\Lambda_{\Gamma}(N(P)).
\]
Here the sum runs over hyperbolic and loxodromic conjugacy classes of $\Gamma$
of norm at most $X$ and $\Lambda_{\Gamma}$ denotes the hyperbolic von Mangoldt
function. That is,
$\Lambda_{\Gamma} (N(P)) = \log N(P_0)$, if $P$ is a power of a primitive
hyperbolic (or loxodromic) conjugacy class $P_0$, and zero otherwise.
The classical bound for the remainder term in~\eqref{eq:pibound} was given by
Sarnak~\cite{sarnak_arithmetic_1983} in 1983. In the arithmetic case,
$\Gamma=\pslzi$, his result says that
\begin{equation}\label{eq:sarnakbound}
    \psi_{\Gamma}(X) = \frac{X^{2}}{2} + E_{\Gamma}(X),\quad
    E_{\Gamma}(X)\ll_{\epsilon}X^{5/3+\epsilon}.
\end{equation}
The estimate~\eqref{eq:sarnakbound} for the error term is actually valid for all
cofinite Kleinian groups, provided that the contribution from possible
small eigenvalues is included in the main term.
Sarnak's pointwise bound~\eqref{eq:sarnakbound} has been improved for the Picard group
in~\cite{koyama_prime_2001,balkanova_prime_2017,balkanova_2018}.
The current best unconditional bound is due to Balkanova
and Frolenkov~\cite{balkanova_2018}, who showed that
\begin{equation}\label{eq:balkanovabound}
    E_{\Gamma}(X)\ll_{\epsilon}X^{\eta+\epsilon},\quad
    \eta=\frac{3}{2}+\frac{103}{1024}.
\end{equation}
By assuming the Lindel\"of hypothesis for quadratic Dirichlet $L$-functions
over Gaussian integers, they obtain $\eta=3/2$.
It is not clear how far this is from the truth (see the discussion in Remarks 1.5 and 3.1 in \cite{balkanova_prime_2017}).

The main result of this paper is that the exponent $3/2+\epsilon$ holds on
average. This is achieved by studying the second moment of the error term.
Namely, we prove the following theorem.
\begin{theorem}\label{intro:thm1}
Let $V\geq Y \gg 1$ and $\eps>0$. Then
\begin{equation}\label{eq:thm1}
    \frac{1}{Y} \int_V^{V+Y} |E_{\Gamma}(X)|^2dX \ll V^{3+\eps} \left(\frac{V}{Y}\right)^{2/3}.
\end{equation}
\end{theorem}

Theorem~\ref{intro:thm1} follows from a short interval second moment estimate for
the spectral exponential sum $S(T, X)$, which is defined as
\[
    S(T, X) = \sum_{0<r_{j}\leq T}X^{ir_{j}},
\]
where $\lambda_{j}=1+r_{j}^{2}$ are the (embedded) eigenvalues of the
Laplace--Beltrami operator $\Delta$ on $M$.
\begin{theorem}\label{intro:thm2}
Let $V\geq Y\gg 1$ and $T\ll V^{1/2-\eps}$. Then
\[
\frac{1}{Y} \int_V^{V+Y} |S(T,X)|^2 dX \ll T^{3+\eps} V^{3/2+\eps} Y^{-1}.
\]
\end{theorem}
The connection between the Prime Geodesic Theorem and $S(T,X)$ is given by
the explicit formula of Nakasuji, see~\eqref{eq:explicit}.

\begin{remark}\label{intro:rmk1}
Note that the bound for arbitrary $Y$ in Theorem~\ref{intro:thm2}
follows by positivity from the estimate over the dyadic interval $[V,2V]$.
Nevertheless, Theorem~\ref{intro:thm2} allows us to prove a nontrivial
result in short intervals in Theorem~\ref{intro:thm1}
since the parameter $T$ can depend on $Y$.
Despite this, we will carry out the proof of Theorem~\ref{intro:thm2} in the
interval $[V,V+Y]$
in order to highlight how the dependence in $Y$ gets absorbed into the
final bound.
\end{remark}

As a corollary of Theorem~\ref{intro:thm1} we recover the pointwise bound
$E_{\Gamma}(X)\ll_{\epsilon}X^{13/8+\epsilon}$ of~\cite[Theorem~1.1]{balkanova_prime_2017}.
Furthermore, our second moment bound~\eqref{eq:thm1} has immediate consequences
analogous to Corollary~1.3 and Equation~(1.3) in~\cite{balkanova_prime_2017},
but we will not write them here explicitly.
Finally, we observe that Theorem~\ref{intro:thm1} implies that the short interval estimate
\[
\frac{1}{V} \int_{V}^{2V} \abs*{\psi_{\Gamma}(X)-\psi_{\Gamma}(X-\eta X)-\eta(1-\eta/2) X^{2}}^{2} dX
\ll_{\epsilon} V^{3+\epsilon},
\]
is valid for all $0\leq\eta\leq 1$.
In other words, the approximation $\psi_\Gamma(X)-\psi_\Gamma(X-\eta X)=\eta(1-\eta/2)X^2$
holds with the error term $O(X^{3/2+\eps})$ in a square mean sense.

A weaker second moment estimate, which is valid for \emph{all}
cofinite $\Gamma$, was obtained in~\cite[Theorem~1.2]{balkanova_prime_2017},
where the authors showed for $V\geq Y \gg 1$ that
\[
    \frac{1}{Y}\int_{V}^{V+Y}\abs{E_{\Gamma}(X)}^{2}dX
\ll
V^{16/5}\left(\frac{V}{Y}\right)^{2/5}(\log V)^{2/5}.
\]
This was proved by using the Selberg trace formula and it is of analogous
strength to Sarnak's bound (see Remark~1.5 in~\cite{balkanova_prime_2017}).
In our proof we will instead use the Kuznetsov trace formula (see~\cite{motohashi_trace_1996,motohashi_trace_1997})
for $\pslzi$, which allows us to get stronger estimates.
A key component of our proof is a careful analysis of integrals involving
multiple Bessel functions. In particular, by relying on exact formulas, we
avoid having to deal with the oscillatory integrals that appear in the proof of
Koyama~\cite{koyama_prime_2001} for the pointwise bound.
We also incorporate some ideas of~\cite{balog_2018}
and~\cite{cherubini_mean_2017} from two dimensions.

The paper is organized as follows. We begin by stating our main tool,
the Kuznetsov formula, in section~\ref{SK}.
Then, in section \ref{S2}, we give a detailed outline
of the proof of Theorem~\ref{intro:thm2} under the assumption of two key estimates,
which are stated as Propositions~\ref{proposition-kloosterman-sums}
and~\ref{proposition-average-rankin-selberg}. In sections~\ref{S3}
and~\ref{S4} we prove these two estimates. Finally, in section~\ref{S5}
we show how to recover Theorem~\ref{intro:thm1} from Theorem~\ref{intro:thm2}.

\section{Kuznetsov formula}\label{SK}
The Kuznetsov trace formula relates the Fourier
coefficients of cusp forms to Kloosterman sums. For Gaussian integers,
Kloosterman sums are defined as
\[
S_{\Q(i)}(n,m,c) = \sum_{a\in (\Z[i]/(c))^\times} e\big(\langle
m,a/c\rangle\big) e\big(\langle n,a^{*}/c\rangle\big),
\]
where $m, n,c\in\mathbb{Z}[i]$, $c\neq 0$; $a^{*}$ denotes the inverse of
$a$ modulo the ideal $(c)$; and $\langle x,y\rangle$ denotes the
standard inner product on $\R^2\cong \C$.
The Kloosterman sums obey Weil's bound~\cite[(3.5)]{motohashi_trace_1997}
\begin{equation}\label{weil}
    |S_{\rationals(i)}(n,m,c)| \leq |(n,m,c)| d(c) N(c)^{1/2},
\end{equation}
which we will use repeatedly.
Here $d(c)$ is the number of divisors of $c$.

\begin{theorem}[Kuznetsov formula
\cite{motohashi_trace_1996,motohashi_trace_1997}]\label{Kuznetsovformula}
Let $h$ be an even function,
holomorphic in $|\Im r|<1/2+\epsilon$, for some $\epsilon>0$,
and assume that $h(r)=O((1+|r|)^{-3-\epsilon})$ in the strip.
Then, for any non-zero $m,n \in {\mathbb Z}[i]$:
\[
D + C = U + S,
\]
with
\[
\begin{split}
D &= \sum_{j=1}^{\infty}   \frac{ r_j \rho_j(n) \overline{\rho_j(m)} }{\sinh \pi r_j}  h(r_j),  \\
C &=  2 \pi \int_{-\infty}^{\infty}     \frac{\sigma_{ir}(n)
      \sigma_{ir}(m)}{ |mn|^{ir} |\zeta_{K}(1+ir)|^2}dr, \\
U &= \frac{\delta_{m,n} + \delta_{m,-n}}{\pi^2} \int_{-\infty}^{\infty} r^2 h(r) dr, \\
S &= \sum_{c \in\zin} \frac{S_{\rationals(i)}(n,m,c)}{|c|^2}
\int_{-\infty}^{\infty}    \frac{ir^2}{\sinh \pi r} h(r) H_{ir} \left(\frac{2 \pi \sqrt{\overline{mn}}}{c}\right) dr,
\end{split}
\]
where $\sigma_s(n) = \sum_{d|n} N(d)^s$ is the divisor function,
\[
H_{\nu} (z) = 2^{-2\nu} |z|^{2 \nu} J_{\nu}^*(z) J_{\nu}^{*}(\overline{z}),
\]
$J_{\nu}$ is the $J$-Bessel function of order $\nu$,
and $ J_{\nu}^*(z) = J_{\nu}(z) (z/2)^{-\nu}$.
\end{theorem}
For the definition of the $\rho_{j}$, see the explanation
after~\eqref{0612:eq001}.
We will also need the power series expansion~\cite[8.402]{gradshteyn2007}
for the $J$-Bessel function:
\begin{equation}\label{eq:jseries}
J_{ir}(z)(z/2)^{-ir} = \sum_{k=0}^\infty \frac{(-1)^k}{k!\Gamma(k+1+ir)}\left(\frac{z}{2}\right)^{2k}.
\end{equation}

\section{Outline of proof of Theorem~\ref{intro:thm2}}\label{S2}

In this section we outline the proof of Theorem~\ref{intro:thm2}.
The result for the sharp sum $S(T,X)$ can be deduced from the corresponding result
for the smooth sum
\begin{equation}\label{def:smoothSTX}
\sum_{r_j} X^{ir_j}e^{-r_j/T}.
\end{equation}
Indeed, if we assume the inequality
\[
\int_V^{V+Y} \Big|\sum_{r_j} X^{ir_j}e^{-r_j/T}\Big|^2 dX \ll T^3 V^{3/2+\eps},
\]
then using a standard Fourier analysis method (see~\cite{iwaniec_prime_1984,luo_quantum_1995})
and the Cauchy--Schwarz inequality we can estimate, for $T<V^{1/2}$ and $Y\leq V$,
\[
\begin{split}
\int_V^{V+Y} \!\!\!
&
|S(T,X)|^2 dX
\\
&\ll
T^\eps \!\! \int_{|\xi|\leq 1} \int_V^{V+Y} \Big|\sum_{r_j} (Xe^{-2\pi\xi})^{ir_j}e^{-r_j/T} \Big|^2 dX \min(T,|\xi|^{-1}) d\xi
+T^{4+\eps}Y
\\
&\ll
T^{3+\eps}V^{3/2+\eps}.
\end{split}
\]
To study the sum in~\eqref{def:smoothSTX}, we approximate $X^{ir_j}e^{-r_j/T}$ by a more regular function
that we borrow from~\cite{deshouillers_iwaniec_1982}, namely
\begin{equation}\label{def:h}
h(r) = \frac{\sinh((\pi + 2i\alpha)r)}{\sinh \pi r},\quad 2\alpha = \log X + i/T,
\end{equation}
which satisfies $h(r) = X^{ir}e^{-r/T} + O(e^{-\pi r})$
\cite{motohashi_trace_1996,motohashi_trace_1997}.
Before applying the Kuznetsov formula, we need to insert the Fourier
coefficients into our spectral sum.
We do this by means of an extra average
and by using the fact that the (normalized) Rankin--Selberg $L$-function has a simple pole at $s=1$
with the residue being an absolute constant.

Consider a smooth function $f$, compactly supported on
$[\sqrt{N},\sqrt{2N}]$,
satisfying $|f^{(p)}(\xi)|\ll N^{-p/2}$ for every $p\geq 0$ and with mass
$\int_0^\infty f(x)dx = N$.
Let $\tilde{f}$ be the Mellin transform of $f$. Then
\begin{equation}\label{0612:eq001}
    \frac{1}{N} \sum_{n,r_j} f(\abs{n}) h(r_j) |v_j(n)|^2
=
\frac{1}{\pi iN} \sum_{r_j} h(r_j) \int_{(3)} \tilde{f}(2s) L(s,u_j\otimes u_j) ds,
\end{equation}
where $L(s,u_j\otimes u_j)$ is the Rankin--Selberg $L$-function
\[
    L(s,u_j\otimes u_j) = \sum_{n\in\zin} \frac{|v_j(n)|^2}{N(n)^s},
\]
where $v_j(n)$ are Fourier coefficients of cusp forms, normalized by the relation
$v_j(n)\sqrt{\sinh(\pi r_j)}=\rho_j(n)\sqrt{r_j}$.
We apply the Kuznetsov formula on the left-hand side in~\eqref{0612:eq001},
while on the right-hand side we move the line of integration to $\Re(s)=1/2$,
picking up the residue at $s=1$. We obtain, for absolute constants $c_1,c_2$,
\begin{equation}\label{eq003}
\begin{split}
    \sum_{r_j} X^{ir_j} e^{-r_j/T}
    &=
    \frac{c_1}{N} \!\! \sum_{n\in\zin} \!\!\!\! f(|n|)
    \mathcal{S}_n(\omega)
    \\
    &+
    \frac{c_2}{N}\int_{(1/2)}
    \tilde{f}(2s) M_1(s) ds
    +
    O(T^2).
\end{split}
\end{equation}
The quantities appearing in~\eqref{eq003} are described as follows.
The term $\mathcal{S}_n(\omega)$ is a weighted sum of Gaussian Kloosterman
sums,
\begin{equation}\label{def:Snomega}
\mathcal{S}_n(\omega)
=
\sum_{c\in\mathbb{Z}[i]\backslash\{0\}} \frac{S_{\mathbb{Q}(i)}(n,n;c)}{N(c)}\omega\left(\frac{2\pi\bar{n}}{c}\right),
\end{equation}
where $\omega$ is the integral transform of $h$ that appears in
Kuznetsov's formula, that is,
\begin{equation}\label{def:omega}
\omega(z) = \int_{-\infty}^{+\infty}\frac{ir^2}{\sinh(\pi
r)}H_{ir}(z)h(r)dr.
\end{equation}
The kernel $H_{ir}(z)$ is given by
\[
H_{\nu} (z) = 2^{-2\nu} |z|^{2 \nu} J_{\nu}^*(z) J_{\nu}^{*}(\overline{z}),
\]
with $J_{\nu}^*(z) = J_{\nu}(z) (z/2)^{-\nu}$, where
$J_{\nu}$ denotes the Bessel function of the first kind.

The term $M_1(s)$ in~\eqref{eq003} is a weighted first moment of Rankin--Selberg
$L$-functions:
\[
M_1(s) = \sum_{r_j} h(r_j) L(s,u_j\otimes u_j).
\]
Note that the integral on the half line in~\eqref{eq003} is absolutely convergent since
$\tilde{f}(2s)\ll N^{1/2} |s|^{-M}$, for arbitrarily large $M$ (when
$\Re(s)=1/2$),
and $L(s,u_j\otimes u_j)$ is polynomially bounded in $|s|$.
Finally, the term $O(T^2)$ in~\eqref{eq003} comes from the identity element and the continuous
spectrum in the Kuznetsov formula.

In sections~\ref{S3} and~\ref{S4} we will prove the following two estimates
that we state as separate propositions.
In order to simplify the exposition,
we assume that $N$ is bounded polynomially
in $X$ and $T$, i.e.
\begin{equation}\label{NTX}
N \ll (TX)^A,
\end{equation}
for some arbitrary $A>0$.
Our final choice of $N$ satisfies this condition and thus~\eqref{NTX} is
not restrictive.
\begin{proposition}\label{proposition-kloosterman-sums}
Let $V\geq Y\gg 1$, and $V^\eps \ll T\ll V^{1/2-\eps}$.
Let $N\geq 1$ be chosen so that \eqref{NTX} holds,
and suppose $n\in\Z[i]$ with $N(n)\asymp N$.
Then
\begin{equation}\label{prop:eq001}
\int_V^{V+Y} |\mathcal{S}_n(\omega)|^2 dX
\ll
(NV^2+T^3Y) (NV)^\eps.
\end{equation}
\end{proposition}
In our proof, the first term in~\eqref{prop:eq001} will be the dominant
one.
Since this term does not depend on $Y$,
the most interesting result is obtained on the full dyadic interval $[V,2V]$.
The same observation was made in Remark~\ref{intro:rmk1} and it applies to the next proposition as well.
\begin{proposition}\label{proposition-average-rankin-selberg}
Let $V\geq Y\gg 1$, and $T\gg 1$. Then
\[
\int_V^{V+Y} |M_1(s)|^2 dX \ll |s|^A T^{6+\eps} V,
\]
for some absolute constant $A$.
\end{proposition}

Let us show that Theorem~\ref{intro:thm2} follows from the above two propositions.
By using the Cauchy--Schwarz inequality and integrating in $X$
in~\eqref{eq003}, we get
\[
\begin{split}
\int_V^{V+Y} \Big|\sum_{r_j} X^{ir_j}e^{-r_j/T}\Big|^2 dX
&\ll
\frac{1}{N} \sum_{N(n)\sim N} \int_V^{V+Y} |\mathcal{S}_n(\omega)|^2 dX
\\
&+
\frac{1}{N} \int_{(1/2)} |s|^{-M} \int_V^{V+Y} |M_1(s)|^2 dX ds
+
O(T^4).
\end{split}
\]
Applying Proposition~\ref{proposition-kloosterman-sums} and
Proposition~\ref{proposition-average-rankin-selberg} yields
\[
\int_V^{V+Y} \Big|\sum_{r_j} X^{ir_j}e^{-r_j/T}\Big|^2 dX
\ll
N^{1+\eps} V^{2+\eps} + T^3 Y (NV)^\eps + \frac{T^{6+\eps}V}{N} + T^4.
\]
We pick $N=T^3V^{-1/2}$
and thus arrive at the inequality
\begin{equation}\label{0512:eq001}
\int_V^{V+Y} \Big|\sum_{r_j} X^{ir_j}e^{-r_j/T}\Big|^2 dX \ll T^{3}
V^{3/2+\eps},
\end{equation}
since $T\leq V^{1/2-\eps}$.
Note that $N\geq 1$ only if $T\geq V^{1/6}$.
For $T<V^{1/6}$, the bound~\eqref{0512:eq001}
follows from the trivial estimate $S(T,X)\ll T^3$. This proves~\eqref{0512:eq001} for every value
of $T\leq V^{1/2-\eps}$, which concludes the proof of
    Theorem~\ref{intro:thm2}.
It remains to prove Propositions~\ref{proposition-kloosterman-sums}
and~\ref{proposition-average-rankin-selberg}.

\section{Second moment of sums of Kloosterman sums}\label{S3}

Next we want to prove Proposition~\ref{proposition-kloosterman-sums}.
In order to do this, we will need to simplify expressions involving
$\omega(z)$ according to the size of $\abs{z}$. We will first prove a
number of  auxiliary lemmas, which are then used in different ranges of
the summation in $\mathcal{S}_{n}(\omega)$.
Throughout this section we shall assume that
$N\geq 1$, $n\in\Z[i]$ satisfies $N(n)\asymp N$,
and that $T$, $X$, $V$ and $Y$ are real numbers satisfying the inequalities
\begin{equation}\label{0412:eq001}
1\ll Y\leq V\leq X\leq V+Y
\quad\text{and}\quad
V^\eps \ll T\leq V^{1/2-\eps}.
\end{equation}
Moreover, we recall the mild assumption~\eqref{NTX} on $N$.

We begin the proof by simplifying the
expression defining $\mathcal{S}_n(\omega)$.
After removing the initial
part of the sum,
we can replace the weight function $\omega$ by
a simpler function $\omega_0$ given by
\begin{equation}\label{def:omega0}
\omega_0(z) = \int_{-\infty}^{\infty} \frac{1}{\Gamma(1+ir)^2} \left|\frac{z}{2}\right|^{2ir} \frac{ir^2h(r)}{\sinh(\pi r)} dr.
\end{equation}
These two simplifications come at the cost of an admissible error term,
as demonstrated in the following lemma.

\begin{lemma}\label{lemma:omega0}
Let $\mathcal{S}_n(\omega)$ be as in~\eqref{def:Snomega},
and let $\omega_0$ be as in~\eqref{def:omega0}. Then
\begin{equation}\label{2211:lemma:eq}
\begin{split}
\mathcal{S}_n(\omega)
=
\mathcal{S}_n^\dagger(\omega_0)
+
O( N^{1/2+\epsilon}T^{1+\eps} ),
\end{split}
\end{equation}
where
\[
\mathcal{S}_n^\dagger(\omega_0)
=
\sum_{N(c)>4\pi^2 N(n)}
\frac{S_{\Q(i)}(n,n,c)}{N(c)} \omega_0\left(\frac{2\pi\bar{n}}{c}\right).
\]
\end{lemma}

\begin{proof}
Let us focus first on the portion of the sum where $N(c)\leq 4\pi^2N(n)$,
i.e.~when the complex number
$z=2\pi\bar{n}/c$ satisfies $|z|\geq 1$.
We start from the definition of $\omega(z)$, see~\eqref{def:omega}, and
apply an integral representation
for the kernel $H_{ir}(z)$ (see~\cite[Equation (2.10)]{motohashi_trace_1997}).
Writing $z=|z|e^{i\theta}$, we have
\[
\omega(z)
=
\frac{8}{\pi^2}
\int_0^{\pi/2} \cos(2|z|\cos\theta\sin\phi)
\int_{-\infty}^{+\infty} r^2 h(r) \cosh(\pi r) K_{2ir}(2|z|\cos\phi) dr \, d\phi.
\]
When $\abs{r}$ is bounded, we estimate $h(r)$ trivially and use the fact
that
$K_{2ir}(x)\ll x^{-1/2}$ for all $r\in\mathbb{R}$ and $x>0$ real.
Thus the integral over $r$
contributes $O(|z|^{-1/2})$ in this range.
Now, for $r$ bounded away from zero, we approximate $h(r)$ by
\begin{equation}\label{1511:eq001}
h(r) = \frac{\cosh((\pi+2i\alpha)r)}{\cosh(\pi r)} + O(e^{-\frac{1}{2}\pi r}),
\end{equation}
and the error contributes again $O(|z|^{-1/2})$.
Note also that, after integrating over $\abs{r}=O(1)$, the fraction
in~\eqref{1511:eq001} is bounded by $O(|z|^{-1/2})$.
The remaining integral reads
\[
I
=
\frac{8}{\pi^2}\int_0^{\pi/2} \cos(2|z|\cos\theta\sin\phi)
\int_{-\infty}^{+\infty} r^2 \cosh((\pi+2i\alpha)r) K_{2ir}(2|z|\cos\phi)dr \, d\phi.
\]
The integral over $r$ can be evaluated exactly by
using the formula~\cite[6.795.1]{gradshteyn2007}. This gives
\[
\int_{-\infty}^{+\infty} r^2 \cosh((\pi+2i\alpha)r) K_{2ir}(2|z|\cos\phi)dr
=
-\frac{\pi}{8} \cdot \frac{\partial^2}{\partial b^2} \exp(-a\cosh b),
\]
where $a:=2|z|\cos\phi$ and $b:=\alpha-i\pi/2$ (in particular, $|\Im(b)|<\pi/2$).
Hence, we arrive at the expression
\[
\begin{split}
I
&=
-\frac{1}{\pi}
\int_0^{\pi/2} \cos(2|z|\cos\theta\sin\phi)
\frac{\partial^2}{\partial b^2} \exp(-a\cosh b) d\phi
\\
&=
-\frac{1}{\pi}
\int_0^{\pi/2} \cos(2|z|\cos\theta\sin\phi)
e^{-2|z|\cos\phi \cosh b}
\\
&\phantom{xxxxxxxxxxx}\times
\left(4|z|^2\cos^2\phi \sinh^{2}b + 2|z|\cos\phi \cosh b\right) d\phi.\rule{0pt}{14pt}
\end{split}
\]
Observe that $\Re(\cosh b)\asymp X^{1/2}T^{-1}$ and $|\cosh b|\asymp |\sinh b|\asymp X^{1/2}$.
In the range $\cos\phi>\log^2T/|z|\Re(\cosh b)$ we bound the integrand in
absolute value.
Since the exponential is $O(T^{-p})$ for arbitrarily large $p$,
the integral contributes $O(|z|^2X T^{-p})$.
On the other hand, when $\cos\phi\leq\log^2T/|z|\Re(\cosh
b)$, we integrate by parts in $\phi$ once.
This gives a factor $1/|z|\cosh b$ from the
exponential and thus the contribution from the integral
is
$O(T^{2+\eps}/|z|X^{1/2})$.
All in all, we have proved that, for $|z|\geq 1$, we can estimate
\[
\omega(z) \ll \frac{1}{|z|^{1/2}} + \frac{|z|^2X}{T^p} + \frac{T^{2+\epsilon}}{|z|X^{1/2}}.
\]
Summing this for $z=2\pi\bar{n}/c$ and $N(c)\leq 4\pi^2 N(n)$,
and using Weil's bound to estimate the Kloosterman sums,
we get a quantity not bigger than
\[
O( N^{1/2+\epsilon} + N^{1+\eps}X T^{-p} + N^{1/2+\epsilon}T^{2+\epsilon}
X^{-1/2} ).
\]
This is absorbed in the error in~\eqref{2211:lemma:eq}
since $V^\eps\ll T\leq V^{1/2-\eps}$ and $N$ satisfies~\eqref{NTX}
(in fact, this is the only place where we use this assumption).

It remains to estimate the portion
of $\mathcal{S}_n(\omega)$ where $N(c)>4\pi^2N(n)$,
i.e.~when $|z|<1$.
In this range we expand the $J$-Bessel functions in the
definition of $\omega$, \eqref{def:omega},
into power series (see~\eqref{eq:jseries}). We get
\begin{equation}\label{1511:eq003}
\begin{split}
\omega(z)
=
\sum_{k_1,k_2\geq 0}
&\frac{(-1)^{k_1+k_2}}{k_1!k_2!2^{2k_1+2k_2}}z^{2k_1}\bar{z}^{2k_2}
\\
&\times
\int_{-\infty}^{\infty} \frac{1}{\Gamma(k_1+1+ir)\Gamma(k_2+1+ir)}\left|\frac{z}{2}\right|^{2ir}\frac{ir^2h(r)}{\sinh(\pi r)}dr.
\end{split}
\end{equation}

By Stirling's formula it follows that, for any $k\geq 0$, we have
\begin{equation}\label{eq:gammaest}
\frac{1}{\Gamma(k+1+ir)}
\ll
e^{k-(k+1/2)\log(1+|r|) + \pi|r|/2}
\ll
\frac{e^{k+\pi|r|/2}}{(1+|r|)^{k+1/2}}
\end{equation}
Using~\eqref{eq:gammaest} and the fact that
$h(r)\ll e^{-|r|/T}$, we can bound all but the initial part of the double
sum in~\eqref{1511:eq003}
(that is, when $k_1+k_2\geq 1$) by
\[
\ll
|z|^2\sum_{k_1+k_2\geq 1} \frac{e^{k_1+k_2}}{k_1!k_2!2^{2k_1+2k_2}}
\int_{-\infty}^{+\infty} e^{-|r|/T}dr
\ll
T|z|^2.
\]
By Weil's bound, summing this for $N(c)>4\pi^2N(n)$ gives a contribution
of at most
$O(N^{1/2+\eps}T)$, which is absorbed in the error term in~\eqref{2211:lemma:eq}.
We are thus left with the term associated to $k_1=k_2=0$, which is precisely
$\omega_0(z)$.
\end{proof}

Next we evaluate $\omega_0$ with an explicit error term.
It turns out that a simple closed formula for $\omega_0$
can be given in terms of the $K$-Bessel function of order zero.
Estimates for $K_0$ and its derivative are collected in the following lemma.
\begin{lemma}\label{lemma:K0}
    Let $K_0(w)$ be the $K$-Bessel function of order zero and let $w\in\C\setminus\{0\}$ with $\Re(w)\geq 0$.
Then
\begin{equation}\label{K0:eq1}
|K_0(w)| \leq \frac{2}{|w|^{1/2}}\exp(-\Re(w)).
\end{equation}
Moreover, setting $f(w)=e^wK_0(w)$, we have
\begin{equation}\label{K0:eq2}
|f'(w)| \leq \frac{1}{|w|^{3/2}}\left(1+\frac{1}{|w|}\right).
\end{equation}
\end{lemma}
\begin{proof}
    The integral representation~\cite[8.432.8]{gradshteyn2007}
\[
K_\nu(w)
=
\sqrt{\frac{\pi}{2w}} \frac{e^{-w}}{\Gamma(\nu+1/2)}
\int_0^\infty e^{-t} t^{-1/2} \left(1+\frac{t}{w}\right)^{\nu-1/2} dt
\]
holds for $|\arg(w)|<\pi$ and $\Re(\nu)>-1/2$.
Since $\Re(w)\geq 0$, the estimate~\eqref{K0:eq1}
follows after bounding the integrand in absolute value.
Multiplying both sides by~$e^w$, differentiating in $w$ and bounding the result
yields~\eqref{K0:eq2}.
\end{proof}

The most important consequence of the simple closed
formula for $\omega_{0}$
is being able to see that $\omega_0(z)\ll|z|^{1/2+\eps}$ as $|z|\to 0$.
Such a decay is guaranteed \emph{a priori} by Kuznetsov's formula,
but is not directly visible in the definition of $\omega_0$
in~\eqref{def:omega0}.
The behaviour of $\omega_0(z)$ for small $z$ is made explicit in the following lemma,
which in turn allows us to replace the infinite sum
over $c$ by a finite sum.

\begin{lemma}\label{lemma:MellinK0}
Let $N\geq 1$ and $n\in\Z[i]$ with $N(n)\asymp N$,
and let $\omega_0$ be as in~\eqref{def:omega0}.
Then
\begin{equation}\label{2311:lemma:eq1}
\omega_0(z) = \frac{iM^2|z|^2X}{\pi} K_0(M|z|X^{1/2}) + O\left(\frac{|z|^{3/2}}{X^{1/4}}\right),
\end{equation}
where $M=e^{-i(\pi/2-1/2T)}$. In particular, we have
\begin{equation}\label{2311:eq003}
\mathcal{S}_n^\dagger(\omega_0)
=
\mathcal{S}_n^{\ddagger}(K_0)
+
O\left(N^{1/2+\eps}X^{1/2+\eps}\right),
\end{equation}
where $\mathcal{S}_n^{\ddagger}$ is a finite weighted sum
of Kloosterman sums given by
\begin{equation}\label{def:ddagger}
\mathcal{S}_n^{\ddagger}(K_0)
=
2iM^2N(n)X
\!\!\!\!\!\!
\sum_{C_1<N(c)\leq C_2}
\!\!\!\!\!\!
\frac{S_{\Q(i)}(n,n,c)}{N(c)^2} K_0\left(\frac{2\pi
M|n|X^{1/2}}{|c|}\right)
\end{equation}
with $C_1=N(n)V/T^2\log^2T$ and $C_2=N(n)V$.
\end{lemma}

\begin{proof}
Let $2A=M|z|X^{1/2}$ and $2B=M|z|X^{-1/2}$.
Then, from
the definition of $h(r)$, together with the relation
\[
\frac{\pi r}{\sinh(\pi r)} = \Gamma(1+ir)\Gamma(1-ir),
\]
we deduce the identity
\begin{align*}
\omega_0(z)
&=
\frac{4A^2}{\pi^2} \int_{(1)} \Gamma(s)^2 A^{-2s} ds
+
\frac{4B^2}{\pi^2} \int_{(1)} \Gamma(s)^2 B^{-2s} ds
\\
&=
\frac{i(2A)^2}{\pi} K_0(2A)
+
\frac{i(2B)^2}{\pi} K_0(2B).
\end{align*}
The second equality follows from~\cite[\S7.3 (17)]{erdelyi_tables_1954}
(see also~\cite[17.43.32]{gradshteyn2007}).
Note that $\Re(2A),\Re(2B)\geq 0$.
Using~\eqref{K0:eq1}, and bounding the exponential
crudely by one, we see that
\begin{equation}\label{2311:eq002}
\left|B^2K_0(2B)\right| \ll |z|^{3/2}X^{-1/4}.
\end{equation}
This proves~\eqref{2311:lemma:eq1}. Also, summing~\eqref{2311:eq002}
over $N(c)>4\pi^2N(n)$ gives a quantity bounded by
\(
O(N^{1/2+\eps}X^{-1/4}),
\)
which is absorbed in the error term in~\eqref{2311:eq003}.
Similarly, from~\eqref{K0:eq1} we see that
\[
\left|A^2K_0(2A)\right| \ll |z|^{3/2}X^{3/4}
\exp\left(-\frac{|z|X^{1/2}}{100T}\right).
\]
Therefore, on summing over $4\pi^2N(n)<N(c)\leq C_1$ and $N(c)\geq C_2$ we obtain a quantity bounded by
\[
O(N^{1/2+\eps}X^{1/2+\eps}).\qedhere
\]
\end{proof}

We have now reduced the problem of estimating
$\mathcal{S}_{n}(\omega)$ to a matter of understanding a finite sum
of Kloosterman sums,
$\mathcal{S}_{n}^{\ddagger}(K_{0})$,
weighted by a $K$-Bessel function of order zero.
\begin{remark}
Notice that if we use the estimate~\eqref{K0:eq1} also in the remaining
range,
$C_1<N(c)\leq C_2$, we obtain $\mathcal{S}_n^\ddagger(K_0)\ll (NTX)^{1/2+\eps}$.
Collecting the errors from Lemma~\ref{lemma:omega0} and Lemma~\ref{lemma:MellinK0},
we see that this contribution
dominates. Therefore we would have
\begin{equation}\label{pointiwse-bound}
\mathcal{S}_n(\omega)\ll (NTX)^{1/2+\eps},
\end{equation}
which recovers the pointwise bound that appears
in~\cite[p.792]{koyama_prime_2001}.
The method in our proof is slightly
different at places
and provides additional details compared to~\cite{koyama_prime_2001}.
Moreover, Lemma~\ref{lemma:MellinK0} bypasses the use of the method of stationary phase,
giving instead a closed formula for the weight function.
\end{remark}

We will now study the second moment of
$\mathcal{S}_n^\ddagger(K_0)$.
By exploiting the oscillation in the weight function
$K_0(M|z|X^{1/2})$, we
can obtain additional decay when integrating in~$X$.
\begin{lemma}\label{lemma:productK0}
Let $M$ be as in Lemma~\ref{lemma:MellinK0}, and let $x_1,x_2$
be positive real numbers satisfying $1\gg x_1,x_2\gg V^{-1/2}$.
Then
\begin{equation}\label{lemma:product:eq}
\int_V^{V+Y} K_0(Mx_1X^{1/2})
\overline{K_0(Mx_2X^{1/2})} dX
\ll
\frac{\min(YV^{-1/2},L^{-1})}{(x_1x_2)^{1/2}},
\end{equation}
where $L=|x_1-x_2|$.
\end{lemma}

\begin{proof}
Consider the function $f(w)=e^wK_0(w)$. From Lemma~\ref{lemma:K0} we have
\begin{equation}\label{f:bound}
    |f(w)|\ll \frac{1}{|w|^{1/2}},\quad |f'(w)|\ll \frac{1}{|w|^{3/2}},
\end{equation}
for $\Re(w)\geq 0$ and $|w|$ bounded away from zero.
The integral in~\eqref{lemma:product:eq} can be written as
\begin{equation}\label{2611:eq001}
\int_V^{V+Y} e^{-X^{1/2}(Mx_1+\overline{M}x_2)} f(Mx_1X^{1/2})\overline{f(Mx_2X^{1/2})}dX.
\end{equation}
Bounding the integrand in absolute value and applying~\eqref{f:bound}
leads to the first term in the minimum in~\eqref{lemma:product:eq}.
The second term in~\eqref{lemma:product:eq} follows from integration by parts
and~\eqref{f:bound}.
\end{proof}
We are now ready to prove Proposition~\ref{proposition-kloosterman-sums}.

\subsection{Proof of Proposition~\ref{proposition-kloosterman-sums}}

By Lemma~\ref{lemma:omega0} and Lemma~\ref{lemma:MellinK0}, we have
\begin{equation}\label{2711:eq004}
\int_V^{V+Y} |\mathcal{S}_n(\omega)|^2 dX
\ll
\int_V^{V+Y} |\mathcal{S}_n^\ddagger(K_0)|^2 dX
+
O(YN^{1+\eps}V^{1+\eps}),
\end{equation}
where $\mathcal{S}_n^\ddagger(K_0)$ is as given in~\eqref{def:ddagger}.
From a dyadic decomposition
and the Cauchy--Schwarz inequality it follows that
we can bound the integral on the right-hand side by
\begin{equation}\label{prop21:eq001}
\begin{split}
&\ll N^{2+\eps} V^{2+\eps} \!\!\!\! \max_{C_1<R\leq C_2}
\int_V^{V+Y} \bigg| \sum_{N(c)\sim R} \!\!\!\!
\frac{S_{\Q(i)}(n,n,c)}{N(c)^2}K_0\left(\frac{2\pi M|n|X^{1/2}}{|c_1|}\right)\bigg|^2 dX
\\
&=
N^{2+\eps} V^{2+\eps} \!\!\!\! \max_{C_1<R\leq C_2} \sum_{c_1,c_2} \frac{S_{\Q(i)}(n,n,c_1)S_{\Q(i)}(n,n,c_2)}{N(c_1c_2)^2}
\\
&\phantom{xxxxxxxx}\times
\int_V^{V+Y} K_0\left(\frac{2\pi M|n|X^{1/2}}{|c_1|}\right)\overline{K_0\left(\frac{2\pi M|n|X^{1/2}}{|c_2|}\right)} dX,
\end{split}
\end{equation}
where the sum over $c_1,c_2$ is restricted to $R<N(c_1),N(c_2)\leq 2R$.
Note that in this range the numbers $x_j=2\pi|n|/|c_j|$ satisfy the inequality $1\gg x_j\gg V^{-1/2}$,
for $j=1,2$, and we can therefore bound the integral in~\eqref{prop21:eq001}
by using Lemma~\ref{lemma:productK0}.
For $x_1=x_2$, i.e.~$|c_1|=|c_2|$, we use the factor $YV^{-1/2}$ in the minimum
in~\eqref{lemma:product:eq}.
This, coupled with Weil's bound for $S_{\Q(i)}(n,n,c)$, leads to the
following estimate for the diagonal part of the sum:
\begin{equation}\label{2711:eq002}
\ll YN^{3/2+\eps}V^{3/2} \!\!\!\!\!\! \sum_{C_1<N(c)\leq C_2} \frac{|(n,c)|^2r_2(N(c))}{N(c)^{5/2-\eps}} \ll T^3Y (NV)^\eps.
\end{equation}
Here we use $r_2(n)$ to denote the
number of ways of writing $n$ as a sum of two squares,
along with the standard estimate $r_2(n)\ll n^\eps$.
For the off-diagonal terms
in~\eqref{prop21:eq001} (when $|c_1|\neq|c_2|$)
we use again Lemma~\ref{lemma:productK0}.
For technical convenience we interpolate
the two bounds in the minimum with the exponents $(\eps,1-\eps)$,
which gives
\[
\min(YV^{-1/2},L^{-1}) \ll V^\eps L^{-1+\eps}.
\]
Inserting this into~\eqref{prop21:eq001}, we can estimate
the double sum over $|c_1|\neq |c_2|$ by
\begin{equation}\label{2711:eq001}
    \ll N^{1+\eps} V^{2+\eps} \!\!\!\!\! \sum_{R\leq
    \ell_1\neq \ell_2\leq 2R}
    \frac{a_{\ell_{1}}a_{\ell_{2}}}{|\ell_{1}-\ell_{2}|^{1-\eps}},
\end{equation}
where $\ell_{j}=N(c_j)$ and the coefficients $a_{\ell}$ are given by
\[
a_{\ell}: = \sum_{N(c)=\ell} \frac{|S_{\Q(i)}(n,n,c)|}{N(c)^{1+\eps}}.
\]
Using the Hardy--Littlewood--P\'olya inequality~\cite[Th.~381, p.~288]{hardy_inequalities_1934}
and Weil's bound~\eqref{weil} for the Kloosterman sums, we can bound~\eqref{2711:eq001} by
\begin{equation}\label{2711:eq003}
    \ll N^{1+\eps} V^{2+\eps} \sum_{\ell\sim R}
    a_{\ell}^{2}
\ll N^{1+\eps} V^{2+\eps} \sum_{c\neq 0} \frac{|(n,c)|^2r_2(N(c))}{N(c)^{1+\eps}} \ll N^{1+\eps}V^{2+\eps},
\end{equation}
uniformly in $R$.
Since the above bound dominates the last term in~\eqref{2711:eq004},
combining~\eqref{2711:eq003} with~\eqref{2711:eq002} we deduce that
\[
\int_V^{V+Y} |\mathcal{S}_n(\omega)|^2 dX \ll N^{1+\eps} V^{2+\eps} + T^3Y
(NV)^\eps,
\]
which is what we wanted to prove.
\qed

\section{Average of Rankin--Selberg $L$-functions}\label{S4}

In this section we prove Proposition~\ref{proposition-average-rankin-selberg}.
As in section~\ref{S3}, we take real numbers $T$, $X$, $V$ and $Y$
satisfying the inequalities in~\eqref{0412:eq001}.
Moreover, we assume that $s$
is a complex number with $\Re(s)=1/2$.

First, we note that $h(r)=X^{ir}e^{-r/T} +O(e^{-\pi r})$.
Therefore,
using the fact that the Rankin--Selberg $L$-function
is bounded polynomially in both $r_j$ and $s$,
we can write
\[
M_1(s) = \sum_{r_j} X^{ir_j} e^{-r_j/T} L(s,u_j\otimes u_j) + O(|s|^A).
\]
We decompose the sum on the right-hand side
into intervals of length $T$ and use the Cauchy--Schwarz
inequality to get
\begin{equation}\label{0412:eq003}
|M_1(s)|^2
\ll
\sum_{m=1}^\infty m^2 \Big|\sum_{(m-1)T<r_j\leq mT} \!\!\!\! X^{ir_j} L_j \Big|^2 + O(|s|^{2A}),
\end{equation}
where we
use the shorthand $L_j=L(s,u_j\otimes u_j)$.
We want to integrate over $X\in [V,V+Y]$
in~\eqref{0412:eq003}. Thus we need to
understand the integral
\[
    I = \int_V^{V+Y} \Big|\sum_{(m-1)T<r_j\leq mT} \!\!\!\! X^{ir_j} L_j \Big|^2 dX.
\]
Opening up the square and integrating directly
yields
\begin{align}\label{eq:Iest}
I
&\ll
V T^\eps e^{-2m} \!\!\!\! \sum_{(m-1)T<r_j,r_k\leq mT} \frac{|L_jL_k|}{1+|r_j-r_k|}
\notag\\
&\ll
V T^\eps e^{-2m} \!\!\!\! \sum_{(m-1)T<r_j\leq mT} \!\!\!\! |L_j|^2 \!\!\!\! \sum_{(m-1)T<r_k\leq mT} \frac{1}{1+|r_j-r_k|}.
\end{align}
The Weyl law (with remainder) on
$\hmodgs$~\cite[Theorem~2]{bonthonneau_weyl_2015} implies
the estimate $\#\{T<r_j\leq T+1\}\ll T^2$ in unit intervals
(see~\cite[(2.1)]{balkanova_prime_2017}).
This gives, for $r_j\leq mT$,
\begin{equation}\label{eq:sumbound1}
\sum_{(m-1)T<r_k\leq mT} \frac{1}{1+|r_j-r_k|} \ll (mT)^{2+\eps}.
\end{equation}
Now, in order to estimate the remaining sum over
$r_{j}$ in~\eqref{eq:Iest}, we use the relation between the Rankin--Selberg $L$-function
and the symmetric square $L$-function, i.e.
\[
L(s,u_j\otimes u_j) = |v_j(1)|^2 \frac{\zeta_{\Q(i)}(s)}{\zeta_{\Q(i)}(2s)}
L(s,\mathrm{sym}^2 u_j),
\]
where $\zeta_{\Q(i)}$ is the Dedekind zeta function of $\Q(i)$.
The bound $v_j(1)\ll r_j^\eps$
(see~\cite[Proposition~3.1]{koyama_prime_2001})
together with the second moment estimate~\cite[Theorem~3.3]{balkanova_prime_2017}
give
\begin{equation}\label{eq:sumbound2}
\sum_{(m-1)T<r_j\leq mT} |L_j|^2 \ll |s|^A (mT)^{4+\eps}.
\end{equation}
After substituting the estimates~\eqref{eq:sumbound1} and~\eqref{eq:sumbound2}
    into~\eqref{eq:Iest}, we finally obtain
\[
I \ll V T^{6+\eps} m^{8+\eps} e^{-2m}.
\]
Using this with~\eqref{0412:eq003} and summing over $m$
leads us to the desired bound.
\qed

\section{Recovering Theorem~\ref{intro:thm1}}\label{S5}

Finally, we show how to recover Theorem~\ref{intro:thm1} from Theorem~\ref{intro:thm2}.
The error term $E_{\Gamma}(X)$ is related to the spectral exponential
sum $S(T,X)$ via the explicit formula. For $\pslzi$, this was proved by
Nakasuji~\cite[Thm. 4.1]{nakasuji_prime_2001}, who showed that
\begin{equation}\label{eq:explicit}
E_\Gamma(X)
=
2\Re\left(\sum_{0<r_j\leq T}\frac{X^{1+ir_j}}{1+ir_j}\right) + O\left(\frac{X^2}{T}\log X\right),
\end{equation}
for $1 \leq T < X^{1/2}$.
Thus, we can estimate the second moment of $E_\Gamma(X)$ as
\[
\frac{1}{Y} \int_V^{V+Y} |E_\Gamma(X)|^2dX
\ll
\frac{1}{Y} \int_V^{V+Y}  \bigg|\sum_{0<r_j\leq T}\frac{X^{1+ir_j}}{1+ir_j}\bigg|^2 dX
+
O\left( \frac{V^4 }{T^2} \log^2 V \right).
\]
Using partial summation, we write the exponential sum as
\[
\sum_{0<r_j\leq T} \frac{X^{1+ir_j}}{1+ir_j}
=
\frac{X}{1+iT} S(T,X) + iX \int_{1}^{T} \frac{S(U,X)}{(1+iU)^2} dU,
\]
and by a repeated use of the Cauchy--Schwarz inequality we obtain
\[
\begin{split}
\int_V^{V+Y}  \bigg| \sum_{0<r_j\leq T}\frac{X^{1+ir_j}}{1+ir_j}\bigg|^2 dX
&\ll
\frac{V^2}{T^2} \int_{V}^{V+Y}  |S(T,X)|^2 dX
\\
&+
V^2 \log T \int_1^T \frac{1}{|1+iU|^3} \int_{V}^{V+Y} |S(U,X)|^2 dX \; dU.
\end{split}
\]
We apply Theorem~\ref{intro:thm2} and bound the right-hand side by
\[
\ll V^{7/2+\epsilon} T^{1+\epsilon} +  V^{7/2+\eps} T^\eps \int_1^T dU
\ll V^{7/2+\epsilon} T^{1+\epsilon}.
\]
Thus
\[
\frac{1}{Y} \int_V^{V+Y} |E_\Gamma(X)|^2dX \ll  V^{7/2+\epsilon} T^{1+\epsilon}
Y^{-1}+  \frac{V^4 }{ T^2} \log^2 V.
\]
Balancing with $T = V^{1/6} Y^{1/3}$ completes
the proof of Theorem~\ref{intro:thm1}.\qed


\end{document}